\documentclass[11pt]{amsart}
\usepackage{graphicx}
\usepackage{amsmath}
\usepackage{enumerate}
\usepackage[top=15mm, bottom=20mm, left=30mm, right=30mm]{geometry}
\newtheorem{theorem}{Theorem}[section]
\newtheorem{lemma}[theorem]{Lemma}
\newtheorem{proposition}[theorem]{Proposition}
\newtheorem{corollary}[theorem]{Corollary}
\newtheorem{definition}[theorem]{Definition}
\newtheorem{example}[theorem]{Example}
\newtheorem*{remark}{Remark}
\newtheorem*{note}{Note}

\newcommand{\Qbar}{\overline{Q}}
\newcommand{\Qstar}{Q^*}
\newcommand{\qstar}{q^*}

\newcommand{\sval}{\psi}
\newcommand{\games}{\mathcal{G}}
\newcommand{\sgames}{\mathcal{SG}}
\newcommand{\fagames}{\mathcal{AG}}

\newcommand{\funs}{\mathcal{F}}

\newcommand{\bij}{\Leftrightarrow}

\newcommand{\csval}{collective semivalue}

\newcommand{\noproof}{\hfill\qed}

\begin{document}
\title{Power measures derived from the sequential query process}
\author{Geoffrey Pritchard}
\author{Reyhaneh Reyhani}
\author{Mark C. Wilson}       
\date{\today}

\keywords{winning coalitions, query, Shapley-Shubik, semivalue} 
\subjclass[2000]{91B12}

\begin{abstract}
We study a basic sequential model for the formation of winning
coalitions in a simple game, well known from its use in defining the
Shapley-Shubik power index. We derive in a uniform way a family of
measures of collective and individual decisiveness in simple games, and
show that, as for the Shapley-Shubik index, they extend naturally to
measures for TU-games. These individual measures, which we call weighted
semivalues, form a class whose intersection with that of the class of
weak semivalues yields the class of all semivalues.

We single out the simplest measure in this family for more investigation,
as it is new to the literature as far as we know. Although it is very
different from the Shapley value, it is closely related in several ways,
and is the natural analogue of the Shapley value under a nonstandard,
but natural, definition of simple game. We illustrate this new measure
by calculating its values on some standard examples.
\end{abstract}

\maketitle

\section{Introduction}
\label{s:intro}

Many authors have discussed the value theory of cooperative TU-games and
its counterpart for simple games, the theory of power measures. 
The material in the present paper arises from a generalization of a particular measure of manipulability of voting rules (below called $\Qbar$) circulated in preprint form by the 
present authors. We realized that our original arguments generalize greatly, and yield a general
construction for TU-games that leads directly to a large class of allocations
including all semivalues. Amongst these, the simplest one yields 
a semivalue with several attractive properties.

\subsection{Our contribution}
\label{ss:contributions} 

We explore the (Shapley-Shubik) sequential model for the formation of
winning coalitions in a simple game, and define
(Section~\ref{s:collective} and~\ref{s:individual}) a family of
decisiveness measures $\Qstar_F$ and individual power measures
$\qstar_F$ for simple games. These measures satisfy many desirable
properties, and extend naturally to the class of TU-games. It turns out
that this construction generates the class of what we call weighted
semivalues. We give several interpretations of these values in
Section~\ref{s:interp}. In this framework the simplest $F$ is affine and
yields a weighted semivalue we call $\qstar_0$, which we single out for
further attention. Although it differs substantially from the
Shapley-Shubik index, it is closely related as we see in
Section~\ref{ss:another}. We illustrate the new measures by applying
them to some well-known games, including games derived from the study of
strategic manipulation of voting rules, our original motivating
examples.

\subsection{Basic definitions}
\label{ss:background} 

The definitions of simple game and TU-game are not entirely standardized.
We use the most common definitions found in the literature. However, in 
Section~\ref{ss:another}, we shall drop some of the assumptions made here.

A \emph{TU-game} on a finite set $X$ is defined by its
\emph{characteristic function} $v:2^X \to \mathbb{R}$, such that 
$v(\emptyset) = 0$. We denote the class of 
all TU-games on $X$ by $\games(X)$. The class of finitely additive TU-games 
consists of those games satisfying $v(S \cup T) + v(S\cap T) = v(S) + v(T)$ 
for all $S,T\subseteq X$, and is denoted $\fagames(X)$. A TU-game is 
\emph{monotonic} if $A\subseteq B \subseteq X$ implies $v(A) \leq v(B)$.

An \emph{allocation} on $X$ is a function $\sval:\games(X) \to \fagames(X)$.
\begin{remark}
An allocation is often called a \emph{value assignment} or simply \emph{value}. An additive game is 
completely specified by the value of $v$ on singletons, so an allocation is just a way of associating 
a nonnegative real number with each player (we prefer not to use a vector, in order to avoid choosing 
an arbitrary ordering on players).

\end{remark}

A \emph{simple game} \cite{TaZw1999} $G = (X, W)$ on a finite set $X$ is
defined by a collection $W$ of subsets of $X$ (called \emph{winning
coalitions}), such that $\emptyset \not\in W$. Equivalently, it is a
TU-game on $X$ where $v$ takes only the values $0$ and $1$ (the value
$1$ corresponding to the property ``winning", whilst coalitions with
value $0$ are called \emph{losing}). Note that we do not require that the game
be nonempty --- that is, we may have $W = \emptyset$. The class of all
simple games on $X$ is denoted $\sgames(X)$ and we define $\sgames$
analogously to $\games$. A special class of game is the (weak) \emph{unanimity game} $\mathcal{U}_S$ 
defined by $S$, where a coalition is winning if and only if it contains $S$. When $|S| = 1$ this is 
called a \emph{dictatorial game}.


\section{The random query process}
\label{s:query}

Let $G = (X,W) \in \sgames$.
Consider the following stochastic process. We choose elements of $X$
sequentially without repetition, at each step choosing uniformly from
the set of elements not yet chosen, until the set of elements seen so
far first becomes a winning coalition. This is the same process
considered by Shapley and Shubik \cite{ShSh1954} in defining their power index
(see Section~\ref{ss:sequential} for more details). 

We first consider the  random variable equal to the number of  queries required.

\begin{definition}
\label{def:Q}
Let $V_1,\ldots, V_n$ be elements sampled without replacement from $X$, where $n=|X|$.
Equivalently, $\pi:= (V_1,\ldots, V_n)$ is a uniformly random
permutation of $X$, representing the order in which elements are to be
chosen. Let
$$ 
Q_\pi(G) = \min \{k: \{V_1,\ldots,V_k\}\text{ contains a winning coalition} \} .
$$
\end{definition}
\begin{remark}
If the game is empty we will not find a winning coalition. In this case
we define $Q_\pi(G)$ to have the value $n+1$. If the game is monotone, 
in Definition~\ref{def:Q} the word ``contains" can be replaced by ``is".
\end{remark}

\begin{definition}
\label{def:Qbar}
The quantity $\Qbar(G)$ is defined to be the expectation of $Q(G)$ with
respect to the uniform distribution on permutations of $X$. 
\end{definition}

\subsection{Non-sequential interpretation}
\label{ss:random subsets}

The sequential nature of the process  is only apparent, once we have
averaged over all possible orders. Thus we ought to be able to find a
representation of $\Qbar(G)$ that does not mention order of players. In order to do 
this, we assume from now on that the game is monotone.

\begin{definition}
\label{def:prob-k}
For each natural number $k$, define the probability measure $m_k$ to be
the uniform measure on the set of all subsets of $X$ of size $k$.
Thus each subset of $X$ of size $k$ is equally likely to be chosen, with
probability $\binom{n}{k}^{-1}$.
\end{definition}

For each natural number $k$, we let $W_k$ (respectively, $L_k$) denote
the set of all winning (respectively, losing) coalitions of size $k$.

\begin{lemma}
\label{lem:Qlem}
For each $k$ with $0\leq k \leq n$,
$$
\Pr(Q(G)\leq k) = \Pr (W_k)
$$
where the latter probability is with respect to $m_k$.

In other words, the probability that we require at most $k$ queries to
find a winning coalition equals the probability that a randomly chosen
$k$-subset is a winning coalition. 
\end{lemma}

\begin{proof}
The event $Q(G)\leq k$ means precisely that the initial subset $A(Q(G),k)$
formed by the first $k$ queries contains a winning subset. Each subset
of $X$ of size $k$ occurs with equal probability $\binom{n}{k}^{-1}$ as
an initial subset of queries of the query sequence, so that $A(Q(G),k)$ is
distributed as a uniform random sample from $X_k$.
\end{proof}

\begin{remark}
The cumulative distribution function of $Q(G)$ can thus be computed
by simply counting the number of winning coalitions of each fixed size.
\end{remark}

We can now derive a simple explicit formula for $\Qbar(G)$.

\begin{lemma}
\label{lem:Qbar}
$$\Qbar(G) = n+1 - \sum_{k=0}^{n} \frac{|W_k|}{\binom{n}{k}}.$$
\end{lemma}
\begin{proof}
For every $\pi$, $Q_\pi(G)$ is at most $n+1$. If $G$ is empty then $W_k$ is
empty for all $k$ and $\Qbar(G) = n+1$, as expected. Otherwise, $W_n$ has a
single element and $Q_\pi(G)$ is at most $n$. Thus by Lemma~\ref{lem:Qlem}
we have
\begin{align*}
\Qbar(G) & = E[Q(G)] \\
& = \sum_{k=0}^{n+1} k \Pr (Q(G) = k) \\
& = \sum_{k=0}^{n} \Pr (Q(G) > k) \\
& = \sum_{k=0}^n \frac{|L_k|}{\binom{n}{k}}\\
& = n+1 - \sum_{k=0}^n \frac{|W_k|}{\binom{n}{k}}
\end{align*}
\end{proof}

\begin{remark}
Note that the summation can start at $k=1$ because the game is
nontrivial. If we allowed trivial games, then the value of the formula for 
$\Qbar(G)$ would be $0$, which agrees with intuition. 
\end{remark}

\section{Changes of variable and collective measures}
\label{s:collective}

The number of random queries made in order to find a winning coalition
seems to us to be a fundamental quantity of a simple game. The quantity
$\Qbar(G)$ intuitively seems to be a measure of inertia or
\emph{resistance} (as discussed in \cite{FeMa1998}): its value is large
if winning coalitions are scarce, and small if they are plentiful. The
rescaled quantity $1 - \Qbar(G)/(n+1)$ looks like an index of what has been
called \emph{complaisance} \cite{FeMa1998, Lind2008} and
\emph{decisiveness}  \cite{Carr2005}. We consider far more general rescalings of 
$\Qbar(G)$, with interesting consequences as will be seen below.

\begin{definition}
\label{def:admissible}
Let $\funs$ be the set of all real-valued functions on the nonnegative integer quadrant $\mathbb{N}^2$.
Let $F\in\funs$ satisfy 
\begin{enumerate}[(i)]
\item $F(n,k)$ is decreasing in $k$ for each fixed $n$.
\item $F(n,0) = 1$ and $F(n,k) = 0$ whenever $k>n$.
\end{enumerate}
We say that $F$ is an \emph{admissible rescaling}.
\end{definition}

\begin{remark}
We do not require that $F$ be decreasing in $n$ for each fixed $k$.

\end{remark}

We shall see below that there is a direct relationship between $F$ and the function $f$ obtained 
as follows.

\begin{lemma}
\label{lem:fF}
There is a bijection  $F \leftrightarrow f$ given by
\begin{align}
\label{eq:fF}
f(n,k) & = \frac{F(n,k) - F(n,k+1)}{\binom{n}{k}}\\
F(n,k) & = \sum_{j=k}^n f(n,j) \binom{n}{j}.
\end{align}
Note that $F$ is admissible if and only if $f$ is nonnegative and 
$\sum_{k=0}^n f(n,k) \binom{n}{k} = 1$.

There is a bijection $F \leftrightarrow \mu$ given by 
\begin{align*}
F(n,k) & = \sum_{j=k}^n \mu(n,j) \\
\mu(n,j) & = F(n,k) - F(n,k+1)
\end{align*}
Note that $F$ is admissible if and only if for each $n$, 
$\mu(n,\cdot)$ is a probability measure on $\{0, \dots, n\}$.
\noproof
\end{lemma}

\begin{remark}
We shall often write $\mu_n(k)$ for $\mu(n,k)$. 
\end{remark}

We now define our candidate for a measure of decisiveness.

\begin{definition}
\label{def:Qstargen}
Let $G = (X,W)\in \sgames$. Define $\Qstar_F(G): \sgames \to \mathbb{R}$ by
$$\Qstar_F(G) = E[F(Q(G))]$$
where the expectation is taken with respect to the uniform distribution on permutations of $X$ as in 
Definition~\ref{def:Qbar}.
\end{definition}

\begin{proposition}
\label{prop:QstarF}

The function $\Qstar_F$ is a decisiveness index on $\sgames$. Explicitly,
$$
\Qstar_F(G) = \sum_{k=0}^n f(n,k) |W_k|
$$

where $f$ and $F$ are linked as in \eqref{eq:fF}.
\end{proposition}

\begin{proof}
Let $f$ be as given in \eqref{eq:fF}. Then
\begin{align*}
\sum_{k=0}^n f(n,k) |W_k| & = \sum_{k=0}^n f(n,k) \binom{n}{k} \Pr(Q(G)\leq k) \\
& = \sum_{k=0}^n \sum_{j=0}^k f(n,k) \binom{n}{k} \Pr(Q(G)=j) \\
& = \sum_{j=0}^n \left(\sum_{k=j}^n f(n,k) \binom{n}{k}\right) \Pr(Q(G)=j)\\
& = \sum_{j=0}^n F(n,j) \Pr(Q(G)=j)\\
& = E[F(Q(G))].
\end{align*}
Now \eqref{eq:fF} and the standing assumptions on $F$ imply that 
$\Qstar_F$ takes values between $0$ and $1$ and these values are attained. Hence $\Qstar_F$ 
is a decisiveness index on $\sgames$ according to the definition in \cite{Wils2011}.
\end{proof}

\begin{example}
\label{eg:Coleman}
Choosing $f(n,k) = 2^{-n}$ yields the Coleman index \cite{Cole1971}. In this case $\mu_n$ is the 
binomial distribution with parameter $1/2$, and 
$$
F(n,k) = 2^{-n} \sum_{j=k}^n  \binom{n}{j}
$$
which equals the probability that a uniformly randomly chosen subset has size at least $k$.
\end{example}

\begin{example}
\label{eg:simplest}

The simplest functional form of the construction above occurs
when  $F(n,k) = 1 - k/(n+1)$, in which case $\mu_n$ is the uniform distribution and
$$
\Qstar_0(G):= \Qstar_{F}(G) = \frac{1}{n+1} \sum_{k=0}^n \frac{1}{\binom{n}{k}} |W_k| 
= 1 - \frac{\Qbar(G)}{n+1}.
$$
\end{example}

There is a close connection between power and decisiveness measures for
simple games and value theory of TU-games \cite{Wils2011}. In view of
that connection, it is natural to generalize to TU-games.

\begin{definition} For each admissible $F$, define a map $\Qstar_F:
\games \to \mathbb{R}$ as follows. For each game $G = (X, v)$,
\label{def:QstarTU}
\begin{align*}
\Qstar_F(G) & :=  \sum_{k=0}^n f(n,k) \sum_{|S|=k, S\subseteq X} v(S)
 = \sum_{S\subseteq X} f(n,|S|) v(S).
\end{align*}
\end{definition}


We usually denote $\Qstar_F(G)$ simply by $\Qstar_F$ when no confusion is likely.

This extended definition of $\Qstar_F$ yields a very general object,
called a \emph{collective value} in \cite{Wils2011}.

\subsection{The self-dual case}
\label{ss:self-dual}

We can derive some special formulae for $\Qstar_F$ in important special
cases. We recall that the \emph{dual} of a TU-game $G=(X,v)$ is the
TU-game $G^* = (X,v^*)$ where $v^*(S) = v(X) - v(X\setminus S)$ for each
$S$. In the case of simple games, winning coalitions become losing, and vice versa,
when passing to the dual. A game is \emph{self-dual} if $v^*(S) = v(S)$ for all $S$.

\begin{proposition}
Let $G=(X,v) \in \games$ and suppose that $F$ satisfies the identity 
\begin{equation}
\label{eq:Fsym}
F(n,k) - F(n,k+1) = F(n,n-k) - F(n,n+1-k).
\end{equation}

If $G$ is self-dual, then $\Qstar_F(G) = v(X)/2$.
\end{proposition}

\begin{remark}
The condition on $F$ says that the probability measure $\mu_n$ is
symmetric on $[0..n]$. For example, the Coleman index satisfies this
property, and Proposition~\ref{prop:formulae} was proved for that
special case in \cite[Proposition 3.4]{Carr2005}. The index $\Qstar_0$
also satisfies the condition. This condition we call \emph{self-duality}
for the following reason. Given a collective value $I$, define another
collective value $I^*$ by $I^*(G) = I(G^*)$. The value is self-dual if
$I=I^*$. It is easily seen that a collective value of the form
$\Qstar_F$ is self-dual if and only if $F$ satisfies the stated
condition.
\end{remark}

\begin{proof}
Let $I = \qstar_F$ satisfy the stated condition and let $G$ be a self-dual game. Then 
\begin{align*}
I(G) + I(G^*) & = \sum_S f(n,|S|) v(S) + \sum_S f(n,|S|) \left[v(X) - v(S)\right] \\
& = \sum_k f(n,k) \binom{n}{k} v(X) \\
& = v(X)
\end{align*}
while
\begin{align*}
I(G) - I(G^*) & = \sum_S f(n,|S|) v(S) - \sum_S f(n,|S|) \left[v(X) - v(S)\right]\\
& = \sum_S f(n,|S|) v(S) - \sum_S f(n,|X\setminus S|) v(X\setminus S) \\
& = 0.
\end{align*}
The result follows by solving for $I(G)$.
\end{proof}

In the case of simple games we can say a little more. Recall that a
simple game is \emph{proper} if the complement of each winning coalition
is losing, and \emph{strong} if the complement of every losing
coalition is winning.

\begin{proposition}
\label{prop:formulae}
Let $G = (N,W)$ be a simple game and suppose that $F$ satisfies the identity \eqref{eq:Fsym}.
\begin{enumerate}[(i)]
\item If $G$ is proper and strong then $\Qstar_F(G) = 1/2$.
\item If $G$ is proper and not strong then $\Qstar_F(G) < 1/2$.
\item If $G$ is strong and not proper then $\Qstar_F(G) > 1/2$.
\end{enumerate}
\end{proposition}

\begin{proof}

For each $k$ we define four types of subset: $D_k$ (respectively
$C_k$) consists of those which are winning, and whose complement is not
(respectively is), whereas $Q_k$ (respectively $P_k$) consists of those
which are losing, and whose complement is not (respectively is). Complementation 
yields a map from $G$ to $G^*$ such that $D_k \bij P_{n-k}, C_k \bij C_{n-k}, Q_k \bij Q_{n-k}$. 
Thus $\Qstar_F(G) + \Qstar_F(G^*) = 1$ as in the proof of the previous proposition, and
\begin{align*}
\Qstar_F(G) - \Qstar_F(G^*) &= \sum_k f(n,k)  \left(|D_k| + |C_k|\right) 
- \sum_k f(n,k) \left(|Q_k| + |P_k|\right)\\
& = \sum_k f(n,k)  \left(|D_k| + |C_k|\right) - \sum_k f(n,k) \left(Q_{k} + D_{k}\right)\\
& = \sum_k f(n,k) \left(|C_k| - |Q_k|\right).
\end{align*}
$G$ is proper if and only if $C_k = 0$ for all $k$, while  $G$ is strong
if and only if $Q_k = 0$ for all $k$ (a simple game is proper and strong
if and only if it is self-dual). The results follow by solving for
$\Qstar_F(G)$.
\end{proof}

\begin{example}
\label{eg:majority}
Next, we consider the unweighted qualified majority voting game. The
winning coalitions are precisely those of size at least $k_0$, for some
fixed $k_0$ (depending on $n$). The value of
$\Qstar_F$ on such a game equals $ \sum_{k=k_0}^n
f(n,k) \binom{n}{k} = F(n,k_0)$. Thus if $n$ is odd and $k_0
= (n+1)/2$ (the straight majority game), $\Qstar_F$ has value $F(n,(n+1)/2$. 
If furthermore $F$ satisfies the symmetry condition \eqref{eq:Fsym}, then 
direct computation shows that $Q^*_F$ takes the value $1/2$. This is to be expected, 
since the game in question is proper and strong.
\end{example}

\section{Individual measures}
\label{s:individual}

In this section we discuss properties of the \emph{marginal function} of $\Qstar_F$,
which we denote $\qstar_F$. Explicitly, $\qstar_{F,i}(G) = \Qstar_F(G) -
\Qstar_F(G_{-\{i\}})$. We first review some properties of semivalues and generalizations.

\subsection{Semivalues and related concepts}
\label{ss:semivalue}

Several classes of allocations have been discussed in the literature. They can be given by axiomatic 
characterizations, but explicit formulae are more useful for our purposes. 

\begin{definition}
\label{def:weak semivalue}
Let $X$ be a finite set. 

A \emph{weighted weak semivalue} on $X$ is an allocation on $X$ that has the form 
$$
\psi_i(v) = \sum_{S\subseteq X} p(S) \left[v(S) - v(S\setminus\{i\}\right]
$$
where $p(S) \geq 0$.

A \emph{weak semivalue} on  $X$ is a weighted weak semivalue for which 
$\sum_{S:i\in S} p(S) = 1$ for each $i\in X$.

\end{definition}

\begin{remark}
The above concepts were introduced in \cite{CaSa2000} in axiomatic terms. An equivalent formulation is 
that an allocation is a weighted weak semivalue if and only if it satisfies the standard axioms of 
Linearity, Positivity, Projection and Balanced Contributions.
\end{remark}

\begin{definition}
\label{def:weighted semivalue}
A \emph{weighted semivalue} on $X$ is a weighted weak semivalue for which $p(S)$ depends only on 
$|S|$, and thus has the form
\begin{equation}
\label{eq:weighted semivalue}
\psi_i(v) = \sum_{k=0}^n f(n,k) \sum_{|S|=k, S \subseteq X} \left[v(S) - v(S\setminus\{i\}\right]
\end{equation}
where $f(n,k) \geq 0$ for all $n,k$.

A \emph{semivalue} on $X$ is a weighted semivalue on $X$ that in addition satisfies the normalization 
condition
\begin{equation}
\label{eq:fnorm}
\sum_{k=1}^n \binom{n-1}{k-1} f(n,k) = 1.
\end{equation}
\end{definition}

\begin{remark}
The definitions have the unfortunate consequence that a semivalue is a
weighted weak semivalue that is both a weak semivalue and a weighted
semivalue! The term \emph{weighted semivalue}  is formally used in the
present paper for the first time, to our knowledge. 
\end{remark}

In \cite{DNW1981} it was proven that semivalues are precisely the allocations satisfying 
the standard axioms Linearity, Positivity, Projection and  Anonymity. The last says that 
if $\pi:X \to X$ is a permutation, then $\sval(\pi i) = \pi \sval(i)$ for each $i\in X$.
Note that because of Anonymity, we may assume that $X = X_n:=\{1,2,\dots, n\}$, 
where $n=|X|$. Let $\games$ denote the union $\bigcup_n \games(X_n)$. 
A \emph{semivalue} on $\games$ is a function that for each $n$ restricts to a 
semivalue on $X_n$. In \cite{DNW1981} it was shown that in addition to the explicit form 
\eqref{eq:fnorm}, the recursion 
\begin{equation}
\label{eq:frec}
f(n,k) = f(n+1,k) + f(n+1,k+1)
\end{equation}
is necessary and sufficient for such an extension. 

The particular weighted semivalue we have in mind is the marginal
function of $\Qstar$. We first show that rescaling is needed in all but
the most trivial cases.

\begin{proposition}
\label{prop:f semivalue}
Let $F$ be an admissible rescaling, let $f$ be related to $F$ as in Lemma~\ref{lem:fF} and let $\psi$ be  
defined as in \eqref{eq:weighted semivalue}. Then 

\begin{enumerate}[(i)]
\item $\psi$ gives a weighted semivalue on $X_n$ for each $n$. 
\item Let $c_n = \sum_{k=1}^n \binom{n-1}{k-1} f(n,k)$. Then $\tilde{\psi}$ defined by 
$\tilde{\psi}_i(v) = \psi_i(v)/c_n$ gives a semivalue on $X_n$ for each $n$.
\item $\psi$ gives a semivalue on $X_n$ if and only if $f(n,k) = 0$ for $1\leq k \leq n-1$ and $f(n,n) = 1$.
\end{enumerate}

\end{proposition}

\begin{proof}
Note that $\psi$ gives a weighted semivalue on $\games_n$ because $F$ is
admissible, hence $f(n,k) \geq 0$ for all $n,k$. Dividing by $c_n$
ensures that the normalization condition \eqref{eq:fnorm} is satisfied.
Finally, suppose that $c_n = 1$ and let $a_{nk} = \binom{n}{k} f(n,k)$.
By admissibility, $a_{nk} \geq 0$ and $\sum_k a_{nk} = 1$. By
hypothesis, $1 = c_n = \sum_k \frac{k}{n} a_{nk}$. Subtracting these two
equalities yields the result.
\end{proof}

\subsection{The marginal function of $\Qstar_F$}
\label{s:qstar}

The marginal function of a decisiveness index is often interpreted as a
power index (the analogue for TU-games is the relationship between a
potential and an allocation \cite{Wils2011, LaVa2005b}). We now explore this direction.

\begin{proposition}
\label{prop:potential}
Let $F$ be an admissible rescaling. Then 
\begin{enumerate}[(i)]
\item $\Qstar_F$ is the potential function of a  function $\qstar_F$ that for each 
$n$ induces a weighted 
semivalue on $\games_n$, given by
\begin{align*}
\qstar_{F,i} & =   \sum_{k=0}^n  f(n,k) \sum_{|S|=k} \left[ v(S) - v(S\setminus\{i\})\right]\\
& = \sum_{S:i\in S} f(n,|S|) D_i(S)
\end{align*}
Here $F$ and $f$ are related as in Lemma~\ref{lem:fF}.
\item
Let 
$$
c_n:= \frac{1}{n}\sum_{k=1}^{n} k \left[ F(n,k) - F(n,k+1) \right].
$$
Then the normalized quantity $\qstar_F/ c_n$ is a semivalue on $X_n$ for each $n$.
\item
$\qstar_F$ is a weighted semivalue on $\games$  if and only if 
$F$ satisfies the recursion identity 
\begin{equation}
\label{eq:Frec}
F(n,k) - F(n,k+1) = \frac{n+1-k}{n+1}  F(n+1,k) 
+ \frac{2k-n}{n+1} F(n+1,k+1) - \frac{k+1}{n+1} F(n+1, k+2).
\end{equation}
If, furthermore, $c_n$ is independent of $n$, then the normalized quantity $\qstar_F/ c_n$ is a semivalue on  
$\games$, given by
$$
\sum_{S:i\in S} \frac{f(n,|S|)}{\Qstar_F(\mathcal{U}_1)} D_i(S)
$$
\end{enumerate}
\end{proposition}

\begin{proof}
The first part follows from Proposition~\ref{prop:f semivalue} because
$\qstar_F$ has exactly the form stated. The other results follow from
the  basic characterization of semivalues, translating the formulae for
$f$ into those for $F$. The coherence recursion \eqref{eq:frec}
translates into
\begin{equation*}
\begin{split}
F(n,k) - F(n,k+1) &= \frac{n+1-k}{n+1} \left[ F(n+1,k) - F(n+1,k+1)\right] \\
&+ \frac{k+1}{n+1} \left[F(n+1,k+1) - F(n+1, k+2) \right].
\end{split}
\end{equation*}

Note that $c_n$ is precisely the value $\Qstar_F (\mathcal{U}_{\{1\}})$ of $\Qstar_F$ on the
dictatorial game with $n$ players. The result now follows by algebraic simplification.
\end{proof}

The construction above is in fact universal.

\begin{proposition}
\label{prop:bijprobsemi}
There is a bijection between probability measures on $\{0,1,\dots, n\}$ and 
weighted semivalues on $\games_n$ given by $\mu_n \leftrightarrow \qstar_F$.
\end{proposition}

\begin{proof}
From Proposition~\ref{prop:potential}, $\qstar_F$ is a weighted semivalue.
Conversely, given a weighted semivalue $\xi$  we define
$\mu_n(k) = \overline{f}(n,k):=f(n,k)/ c_n$ where $c_n:=\sum_k f(n,k)
\binom{n}{k}$. Then defining $F$ by \eqref{eq:fF} applied to
$\overline{f}$ we have $\qstar_F = \xi$. 
\end{proof}

\begin{remark}
Note that a weighted semivalue satisfies the normalization condition for a
semivalue if and only if for each $n$, the mean of $\mu_n$ is exactly
$n$. 

The recursion \eqref{eq:frec} translates to 
$$
\mu_n(k) = \left[ 1 - \frac{k}{n+1} \right] \mu_{n+1}(k) + \frac{k+1}{n+1} \mu_{n+1}(k+1).
$$
\end{remark}

\begin{example}
\label{eg:Banzhaf}
For the Coleman index, the marginal function is given by $f(n,k) = 2^{-n}$. The associated semivalue 
(obtained by dividing by $\Qstar_F(\mathcal{U}_1) = 1/2$) is the Banzhaf value.
\end{example}

\begin{example}
\label{eg:Shapley}
We now consider the Shapley value, given by 
\begin{align*}
\sigma_i(G) & = \sum_{\emptyset \subset S \subseteq X} 
\frac{(n-|S|)!(|S|-1)!}{n!} \left[ v(S) - v(S\setminus\{i\}) \right] \\
& = \sum_{\emptyset \subset S \subseteq X} 
\frac{|S|^{-1}}{\binom{n}{|S|}} \left[ v(S) - v(S\setminus\{i\}) \right] \\
& = \sum_{k=1}^n  \frac{1}{k\binom{n}{k}}  \sum_{S\subseteq X, |S| = k}
\left[ v(S) - v(S\setminus\{i\}) \right]
\end{align*}
It is the semivalue associated to $\qstar_F$ where
$F(n,0) = 1$ and for $k\geq 1$
\begin{equation*}
\label{eq:ssF}
F(n,k)  = \frac{H_n - H_{k-1}}{H_n} :=\frac{\sum_{j=k}^n \frac{1}{j}}{\sum_{j=1}^n \frac{1}{j}}.
\end{equation*}
where $H_n$ denotes as usual the $n$th harmonic number.
\end{example}

We single out the simplest case for special mention. Recall that a
\emph{regular} semivalue is one for which the weights $f(n,k)$ for $1
\leq k \leq n$ are all nonzero.

\begin{proposition}
\label{prop:qstar0}
The formula
\begin{equation}
\label{eq:qstar0}
\frac{\qstar_{0,i}}{\Qstar_0(\mathcal{U}_1)} = \frac{2}{n+1}\sum_{k=0}^n\sum_{S\subseteq X, |S|=k} 
\frac{\left[ v(S) - v(S\setminus\{i\})\right]}{\binom{n}{k}} 
\end{equation}
defines a regular semivalue on $\games$.
\end{proposition}

\begin{proof}
This follows from Proposition~\ref{prop:potential} and
Proposition~\ref{prop:formulae}, because $F(n,k) = 1 - k/(n+1)$ satisfies the
identity \eqref{eq:Fsym} and so $\qstar_0$ is self-dual.
\end{proof}

\section{Interpretation of the measures}
\label{s:interp}

The collective value $\Qstar_F$ can be easily interpreted as a
decisiveness index on simple games, which gives the probability of
finding a winning coalition when coalitions are sampled first by
choosing size according to $\mu_n$ and then choosing a coalition of that
size uniformly at random. The individual value $\qstar_F$ can be
interpreted in the usual way (not without controversy) as a power index
\cite{Wils2011}. In this section we consider some other interpretations.

\subsection{Coalition formation}
\label{ss:coalform}

Consider the following model of coalition formation  \cite{LaVa2008b}. We assume
that each possible coalition (subset $S$ of $X$) forms with probability $p(S)$, 
and that only one coalition $S$ will form. Consider the following two expectations. 
First, the \emph{ex ante} expected marginal contribution of $i$ to $S$ is
$$
E[D_i(S)] := E[v(S) - v(S\setminus\{i\}) ] 
= \sum_{S:i\in S} p(S) \left(v(S) - v(S\setminus\{i\}) \right).
$$
The \emph{ex interim} expected marginal contribution of $i$ to $S$, conditional on $i\in S$, is
$$
\Phi_i(v,p):= E[D_i(S) \mid S\ni i] = \frac{E[D_i(S)]}{\Pr(S\ni i)}.
$$
Then \cite[Proposition 3]{LaVa2008b} the maps $\Phi_i(\cdot,p)$ are in bijection with the 
set of all probability distributions on $2^X$. Furthermore the map corresponding to $p$ is 
precisely the \emph{weighted weak semivalue} given by
$$
\Phi_i(v,p) = \sum_{S:i\in S} \frac{p(S)}{\sum_{T:i\in T} p(T)} D_i(S).
$$

Note that  $\Phi$ is a weak semivalue if and only if $\sum_{T:i\in T}
p(T)$ does not depend on $i$, and a semivalue if and only, if in
addition, $p(S)$ depends only on $|S|$. Thus for weak semivalues, the ex
ante and ex interim marginal contribution of $i$ to $S$ is the same.

We now apply the above framework to our measures $\qstar_F$.

\begin{proposition}
Let $F$ be an admissible rescaling. Then under the coalition formation model above, 
$\qstar_{F,i}$ gives the ex ante expected contribution of $i$ to $S$, while the associated semivalue gives 
the ex interim expected marginal contribution of $i$ to $S$, conditional on $i\in S$.
\noproof
\end{proposition}

\begin{remark}
The special case of the Shapley value was discussed in
\cite[Proposition 2]{LaVa2008b}, where essentially the same formula was derived. 
\end{remark}

\subsection{Sequential interpretation}
\label{ss:sequential}

The Shapley value $\sigma$ can be defined sequentially as follows. Given
a game $G=(X,v)$, follow the query process and at each step award $v(S)
- v(S\setminus\{i\})$ to $i$, where $S$ is the set of elements queried
so far and $i$ is the last element queried. The expected value with
respect to the uniform distribution on permutations of $X$ is the
Shapley value $\sigma_i(G)$. For simple games, this means that we award
$1$ point to $i$ whenever $i$ is \emph{pivotal}, and 0 otherwise.

We can generalize this interpretation to measures of the form
$\qstar_F$. We follow the query process, and award $k\mu_n(k)$ to the
pivotal element if the process stops at $k$ queries. For example, for
$\qstar_0$ this means awarding $k/(n+1)$, the fraction of the maximum
possible number of queries made (a possible interpretation is that we
offer a higher price to pivotal elements to reveal themselves as we
repeatedly fail to find them) . This corresponds to the nonsequential
formula for $\qstar_F$ using the analogous computation to that for the
Shapley value above.

\subsection{Another model of semivalues}
\label{ss:another}

The standard model of simple game requires that the empty coalition
never be winning (this is a consequence of the standard assumption that
$v(\emptyset) = 0$ for all TU-games). For reasons of mathematical
elegance, at least, we prefer that the class of simple games be closed
under duality. This then requires that the grand coalition is always
winning, which is a reasonable assumption for weighted voting games.
However for many applications, such as to the study of manipulation, 
empty games naturally arise. Thus, by duality, we should admit
trivial games to the class of simple games. This approach is followed,
for example, by Taylor and Zwicker \cite{TaZw1999}. 

We denote by $\sgames^+(X)$ the class of simple games in this extended sense. 
The analogue for TU-games we denote $\games^+(X)$.

In this new model, the axiomatic definition of semivalue remains the
same, since the efficiency, anonymity, positivity and projection
axioms still make sense. However the characterization of \cite{DNW1981} changes
slightly.

\begin{proposition}
\label{prop:ext sval}
A function $\sval$ is a semivalue on $\games^+(X)$ if and only if it has the form
$$
\sval(i) = \sum_k f(n,k) \sum_{S\subseteq X, |S| = k} 
\left[v(S) - v(S\setminus\{i\})\right]
$$
where 
$f(n,k) \geq 0$ and $f$ satisfies the identity

$$
\sum_k f(n,k) \binom{n}{k} = 1.
$$
Such a function extends to a semivalue on $\games^+$ if and only if $f$ also satisfies the identity
$$
f(n,k) = f(n,k+1) + f(n+1,k+1).
$$
\end{proposition}

\begin{proof}
The proof of \cite{DNW1981} still works, with the only change being that the normalization 
condition is slightly different. This is because the vector space of all games now includes the 
empty game in the standard basis of unanimity games and so has dimension one more than 
before.
\end{proof}

\begin{corollary}
Let $F$ be admissible. Then $\qstar_F$ is a semivalue on $\games^+(X)$.
\end{corollary}

\begin{remark}
For example, $\qstar_0$ is the exact analogue of the Shapley value in
this new context. The Shapley value is the semivalue that has equal
weight on all coalition sizes from $1$ to $n$, hence the formula $f(n,k)
= [n \binom{n-1}{k-1}]^{-1}$, whereas $\qstar_0$ has equal weight on all
coalition sizes from $0$ to $n$.
\end{remark}

\section{The measures $\Qstar_0$ and $\qstar_0$}
\label{s:Q0}

We recall that $\Qstar_0(G)$ is the probability that a randomly sampled
coalition is winning, when a coalition size is chosen uniformly and
condition on that a particular coalition is chosen. Alternatively, it
gives the expected fraction of the maximum possible number of queries
saved when players are sampled uniformly at random without replacement
until a winning coalition has been found.

\begin{example}
\label{eg:size 4}
We present the values of $\Qstar_0$ for all simple games with $4$
players (up to isomorphism). Because of anonymity, we may call the
players $1,2,3,4$ and we define each game by listing its minimal winning
coalitions in the obvious way (the last line corresponds to the empty
coalition). Table~\ref{t:n=4} is the analogue of the table in
\cite{Carr2005} for the Coleman index $C$, which we also list for
comparison. We have reordered some rows from the table in \cite{Carr2005}, 
by permuting some rows with equal values of $C$. These rows are marked with $^*$.

Note that the values of $C$ always (weakly) decrease going down the
column, as do those of $\Qstar_0$. Thus $C$ and $\Qstar_0$ never
disagree on the relative decisiveness of two games when they both agree
that two games are not equally decisive. Even though the range of the
values of $C$ (excluding the empty game) is much larger than the range
of values of $\Qstar_0$, the latter never has equal values on two games
when the former does not, but the former sometimes has equal values when
the latter does not. Thus $\Qstar_0$ appears to discriminate better
between games.

\begin{table}
\begin{tabular}{|l|l|l|}
\hline
$W^m$ & $C$ & $\Qstar_0$ \\ \hline
$1;2;3;4$ &$0.9375$  & $0.8000$ \\ \hline
$1;2;3$ &$0.8750$ & $0.7500$\\ \hline
$1;2;34$ &$0.8125$ & $0.7000$ \\ \hline
$1;2$ &$0.7500$ & $0.6667$ \\ \hline
$1;23;24;34$ &$0.7500$ & $0.6500$ \\ \hline
$1;23;24$ &$0.6875^*$ &$0.6067$  \\ \hline
$12;13;14;23;24;34$ &$0.6875^*$& $0.6000$  \\ \hline
$12;13;14;23;24$ &$0.6250$ &$0.5833$  \\ \hline
$1;23$ & $0.6250$& $0.5500$ \\ \hline
$1;234$ &$0.5625^*$ &   $0.5500$ \\\hline
$12;13;14;23$ &$0.5625^*$ &$0.5333$  \\ \hline
$12;13;24;34$ &$0.5625^*$  &$0.5333$  \\ \hline
$1$ & $0.5000$ &$0.5000$  \\ \hline
$12;13;23$ & $0.5000$& $0.5000$ \\ \hline
$12;13;24$ & $0.5000$&  $0.5000$\\ \hline
$12;13;14;234$ & $0.5000$&$0.5000$  \\ \hline
$12;34$ & $0.4375^*$&  $0.4667$\\ \hline
$12;13;234$ & $0.4375^*$&$0.4667$  \\ \hline
$12;13;14$ &$0.4375^*$ &  $0.4500$\\ \hline
$12;13$ & $0.3750$&  $0.4167$\\ \hline
$12;134;234$ &$0.3750$ &$0.4033$  \\ \hline
$123;124;134;234$ &$0.3125^*$ &$0.4000$  \\ \hline
$12;134$ &$0.3125^*$ &  $0.3833$\\ \hline
$123;124;134$ &$0.2500$ &  $0.3500$\\ \hline
$12$ &$0.2500$ &  $0.3333$\\ \hline
$123;124$ &$0.1875$ & $0.3000$ \\ \hline
$123$ &$0.1250$ &  $0.2500$\\ \hline
$1234$ &$0.0625$ &  $0.2000$\\ \hline
 & $0.0000$ &$0.0000$  \\ \hline
\end{tabular}
\caption{Values of $\Qstar_0$ and $q_0$ for simple games with $4$ players.}
\label{t:n=4}
\end{table}

\end{example}

The next type of example was our original motivation for the study of
$\Qstar_0$. The quantity $\Qbar$ seems to us a compelling way to measure
the effort taken by an external agent to change the outcome of the
election via manipulation. Assuming that the voting situation is known
but not the complete profile (we may know from polling data how many
voters of each type there are, but not their identity), the agent
incurs a unit cost to determine each voter's type. The sequential
model occurs naturally here.

\begin{example}
\label{eg:manip}
Consider a 3-candidate election where voters submit complete and total
preference orders. Suppose that the voting situation is as follows: 2
voters have preference order $abc$, 1 has $bac$ and another has $cba$. A
manipulation is a change of vote by some subset that causes a preferred
outcome for those voters, assuming the other voters continue to vote
sincerely. We define a winning coalition to be a subset containing a
subset that can manipulate coalitionally. Here for concreteness we break
ties uniformly at random, and assume risk-averse manipulators (see
\cite{PrWi2007} for more details). 

Using a positional scoring rule that awards $1, \alpha, 0$ points to the
first, second, third ranked candidates, we see that the scores of $a, b,
c$ respectively are $2+\alpha, 1+3\alpha, 1$. 

There is no manipulating coalition which can make $c$ win, since the last voter cannot 
help $c$ overtake both $a$ and $b$, while the other voters have no incentive to do so.

If $\alpha \leq 1/2$ then $a$ wins (solely, unless $\alpha = 1/2$ in
which case $a$ and $b$ tie). The $cba$ voter can change to $bca$, and
this allows $b$ to win, so is preferred by that voter. Thus a
manipulating coalition of size $1$ exists. Furthermore, if $\alpha \geq
1/3$, then the $bac$ voter also has the power to make $b$ win by voting
$bca$. It follows that for $1/3 \leq \alpha \leq 1/2$, the winning
coalitions are those containing either of the last two voters, while for
$0\leq \alpha < 1/3$, the winning coalitions are those containing the
last voter. In the former case, a winning coalition is found by the
random query process after $3$ queries for $4$ of the $24$ possible query
sequences, after $2$ queries for $8$ query sequences and after $1$ query for
$12$ query sequences. Thus $\Qbar = 40/24$ and $\Qstar_0 = 1 - \Qbar/5 =
2/3$. Similarly, in the latter case the relevant voter is found after
$1$,$2$,$3$, or $4$ queries in each of $6$ query sequences, leading to  $\Qbar =
60/24$ and $\Qstar_0 = 1/2$.

If $\alpha > 1/2$ then $b$ is the sole winner. Either of the $abc$
voters can make $a$ win by switching to $acb$. In this case we have by
the same argument as above that $\Qstar_0 = 2/3$. As described in detail
above, these values of $\Qstar_0$ give a measure of the probability of
finding a manipulating coalition when coalitions are sampled according
to a particular probability distribution.

In the case $0\leq \alpha < 1/3$, the individual measure $\qstar_0$ has
the value $0$ for all but the last voter, and $1/2$ for that voter. When
$1/3 \leq \alpha \leq 1/2$, the values are $0$ for the first two voters,
and $2/15$ for each of the last two, whereas when $1/2 < \alpha$, the
roles of the first pair and last pair are reversed. These numbers
represent the ex ante probability of being critical in a winning
coalition, if coalitions form according to the particular probabilistic
model used here. 

\end{example}

\subsection{A bargaining model}
\label{ss:bargain}

Laruelle and Valenciano \cite{LaVa2008c} present a model of bargaining
intended to help give a noncooperative foundation to the theory of power
indices and values. They discuss a setup where a proposer suggests an
initial allocation of payoffs to a winning coalition containing the
proposer. Let $p$ be a map that for each set $X$ of players, takes each
simple game on $X$ to a probability distribution over $X\times 2^X$. The
idea is that $p_G(i,S)$ is the probability that $i$ will be the proposer
with the support of $S$. They impose a dummy axiom which leads to the
condition that $p_G(i,S) = 0$ unless $i$ swings $S$. Anonymity is also a
reasonable assumption.

They discuss nonsequential (``first choose $S$, then $i$") and
sequential (``choose $i$ and $S$ simultaneously") approaches. In the
former case, given a probability distribution over coalitions where the
probability of $S$ depends only on $|S|$, in the first round we choose
$S$ and then choose choosing $i\in S$ uniformly at random. If $i$ swings
$S$, then $i$ is the proposer, otherwise we repeat rounds until a
proposer is found). Thus we are in the arena of weighted semivalues and
we write $p(n,s)$ for $p(S)$ when $|X| = n$.

\begin{proposition}
\label{prop:Qstarbarg}
Let $F$ be admissible and consider the nonsequential protocol above, where $S$ is chosen according
to the probability given by $\mu_F$. Let $\pi_i$ be the probability that $i$ is eventually chosen as 
the proposer, and let $r_i$ be the probability that  $i$ is chosen as proposer in the first step. 
Then there is $c$ so that $\pi_i = c r_i$ for all $i$.

In particular, for $F=F_0$, $\pi_i$ equals the Shapley-Shubik index of $i$.
\end{proposition}

\begin{proof} 
Let $r_i$ denote the probability that  $i$ is chosen as proposer in the first step. Then  
$\pi_i = r_i/\sum_i r_i$. Now 
\begin{align*}
r_i & = \sum_{S:S\in W, i\in S, S\setminus\{i\} \not\in W} \frac{p(S)}{|S|} \\
& = \sum_{S\neq \emptyset} p(S) \frac{D_i(S)}{|S|}\\ 
& = \sum_{k=1}^n \frac{f(n,k)}{k} \sum_{|S|=k}  D_i(S)\\
\end{align*}

When $F=F_0$, $f(n,k)/k$ is precisely $[n+1]^{-1}$ times the coefficient
found in the Shapley-Shubik index, namely $[k\binom{n}{k}]^{-1}$. Thus
$r_i$ equals $[n+1]^{-1}$ times the Shapley-Shubik index $\sigma_i$ of
$i$. Thus since $\sum_i r_i$ is independent of $i$ (in fact it equals
$[n+1]^{-1}$), the ratios $\pi_i/r_i$ are constant, and so the
normalized probability $\pi_i$ equals $\sigma_i$.
\end{proof}

\begin{remark}
This result was stated (in other words) without proof in \cite[p. 124]{LaVa2008c}.
\end{remark}

In the sequential scenario, the query process seems a very natural one. If
the pivotal voter is always chosen as the proposer, then the probability
of being the proposer is again the Shapley-Shubik index. However other
choices are possible. For example, for each admissible $F$, if we weight
the pivotal voter by $k\mu_n(k)$ every time it is pivotal in position
$k$, and then compute the overall probability accordingly, we obtain the
normalized version of $\qstar_F$.

\subsection*{Acknowledgement}
We thank the anonymous referee for several insightful comments that
helped to improve the presentation of this paper.

\bibliographystyle{plain}
\bibliography{power}
\end{document}